\documentclass{elsarticle}

\usepackage{amssymb,amsmath}

\newcommand{\F}{\mathbb{F}}
\newcommand{\Z}{\mathbb{Z}}
\newcommand{\N}{\mathbb{N}}

\newcommand{\bq}{\mathbf{q}}

\DeclareMathOperator{\ord}{ord}
\DeclareMathOperator{\PCN}{PCN}
\DeclareMathOperator{\CN}{CN}

\newtheorem{theorem}{Theorem}[section]
\newtheorem{lemma}[theorem]{Lemma}
\newtheorem{proposition}[theorem]{Proposition}
\newtheorem{conjecture}[theorem]{Conjecture}
\newtheorem{corollary}[theorem]{Corollary}

\newdefinition{definition}[theorem]{Definition}

\newproof{remark}{Remark}
\newproof{proof}{Proof}

\begin{document}

\title{On the existence of primitive completely normal bases of finite fields}
\author[crete]{Theodoulos Garefalakis}
\ead{tgaref@uoc.gr}

\author[sabanci]{Giorgos Kapetanakis\corref{cor1}}
\ead{gnkapet@gmail.com}

\cortext[cor1]{Corresponding author}

\address[crete]{Department of Mathematics and Applied Mathematics, University of Crete, Voutes Campus, 70013 Heraklion, Greece}

\address[sabanci]{Sabanci University, FENS, Orhanli-Tuzla, 34956 Istanbul, Turkey}

\begin{abstract}
Let $\F_q$ be the finite field of characteristic $p$ with $q$ elements and $\F_{q^n}$ its extension of degree $n$. We prove that there exists a primitive element of $\F_{q^n}$ that produces a completely normal basis of $\F_{q^n}$ over $\F_q$, provided that $n=p^{\ell}m$ with $(m,p)=1$ and $q>m$.
\end{abstract}


\maketitle

\section{Introduction}\label{sec:intro}

Let $\F_q$ be the finite field of cardinality $q$ and $\F_{q^n}$ its extension of degree $n$, where $q$ is a prime power and $n$ is a positive integer.
A generator of the multiplicative group $\F_{q^n}^*$ is called \emph{primitive}. Besides their theoretical interest, primitive
elements of finite fields are widely used in various applications, including
cryptographic schemes, such as the Diffie-Hellman key exchange
\cite{diffiehellman76}.

An \emph{$\F_q$-normal basis} of $\F_{q^n}$ is an $\F_q$-basis of $\F_{q^n}$ of the form $\{ x, x^q, \ldots , x^{q^{n-1}} \}$ and the element $x\in\F_{q^n}$ is called \emph{normal over $\F_q$}. These bases bear computational advantages for finite field
arithmetic, so they have numerous applications, mostly in coding theory and cryptography. For further information we refer to \cite{gao93} and the references therein.

It is well-known that primitive, see \cite[Theorem~2.8]{lidlniederreiter97}, and normal, see \cite[Theorem~2.35]{lidlniederreiter97}, elements exist for every $q$ and $n$. The existence of
elements that are simultaneously primitive and normal is also well-known.
\begin{theorem}[Primitive normal basis theorem]\label{thm:pnbt}
Let $q$ be a prime power and $n$ a positive integer. There exists some $x \in
\F_{q^n}$ that is simultaneously primitive and normal over $\F_q$.
\end{theorem}
Lenstra and Schoof \cite{lenstraschoof87} were the first to
prove Theorem~\ref{thm:pnbt}. Subsequently, Cohen
and Huczynska \cite{cohenhuczynska03} provided a computer-free proof with the
help of sieving techniques. Several generalizations of this have also been investigated
\cite{cohenhachenberger99,cohenhuczynska10,hsunan11,kapetanakis13,kapetanakis14}.

An element of $\F_{q^n}$ that is simultaneously normal over $\F_{q^l}$ for all $l\mid n$ is called \emph{completely normal over $\F_q$}. The existence of such elements for any $q$ and $n$ is well-known \cite{blessenohljohnsen86}. Morgan and Mullen \cite{morganmullen96} conjectured that for any $q$ and $n$, there exists a primitive completely normal element of $\F_{q^n}$ over $\F_q$.
\begin{conjecture}[Morgan-Mullen]\label{conj:mm}
Let $q$ be a prime power and $n$ a positive integer. There exists some $x \in
\F_{q^n}$ that is simultaneously primitive and completely normal over $\F_q$.
\end{conjecture}
In order to support their claim, Morgan and Mullen provide examples for such elements for all pairs $(q,n)$ with $q\leq 97$ and $q^n<10^{50}$, see \cite{morganmullen96}. This conjecture is yet to be completely resolved. Partial results, covering certain types of extensions have been given, see \cite{hachenberger13} and the references therein. Recently, Hachenberger \cite{hachenberger16}, using elementary methods, proved the validity of Conjecture~\ref{conj:mm} for $q\geq n^3$ and $n\geq 37$. 
In this work, we begin by proving the following theorem in Section~\ref{sec:proof}.
\begin{theorem}\label{thm:our1}
Let $n\in\N$ and $q$ a prime power with $q\geq n$, then there exists a primitive completely normal element of $\F_{q^n}$ over $\F_q$.
\end{theorem}
Then, we extend Theorem~\ref{thm:our1} by pushing our methods further and obtain the following generalization in Section~\ref{sec:proof-0}.
\begin{theorem}\label{thm:our0}
Let $q$ a power of the prime $p$ and $\ell,m\in\Z$ with $\ell\geq 0$, $m\geq 1$, $(m,p)=1$. 
If $n=p^{\ell}m$ and $m<q$, then there exists a primitive completely normal element of $\F_{q^n}$ over $\F_q$.
\end{theorem}
Our method is based on the work of Lenstra and Schoof \cite{lenstraschoof87}. In particular, we give sufficient conditions, for our existence results, that are progressively easier to check, but harder to satisfy. This way, me manage to prove our theorems theoretically for all pairs $(n,q)$ that satisfy the stated conditions, with the exception of 18 resilient pairs. For all those pairs, however, examples of primitive and completely normal elements have been given in \cite{morganmullen96}. For the reader's convenience, these pairs are displayed in Table~\ref{table:exceptions}.
\section{Preliminaries}\label{sec:prelim}
Before we move on to our results, we note that, in addition to the special cases mentioned in \cite{hachenberger13}, the case when $\F_{q^n}$ is completely basic over $\F_q$ can be excluded from our calculations. Namely, $\F_{q^n}$ is \emph{completely basic over $\F_q$} if every normal element of $F_{q^n}$ is also completely normal over $\F_q$ and it is clear that in that case, Theorem~\ref{thm:pnbt} implies Conjecture~\ref{conj:mm}. Furthermore, we can characterize such extensions using the following, see \cite[Theorem~5.4.18]{hachenberger13} and, for a  proof, see \cite[Section~15]{hachenberger97}.
\begin{theorem}\label{thm:cbe}
Let $q$ be a power of the prime $p$. $\F_{q^n}$ is completely basic over $\F_q$ if and only if for every prime divisor $r$ of $n$, $r \nmid \ord_{(n/r)'}(q)$, where $(n/r)'$ stands for the $p$-free part of $n/r$ and $\ord_{(n/r)'}(q)$ for the multiplicative order of $q$ modulo $(n/r)'$.
\end{theorem}
With the above and Theorem~\ref{thm:pnbt} in mind, it is straightforward to check the validity of the following:
\begin{corollary}\label{cor:cbe}
Let $q$ be a power of the prime $p$ and $n=p^\ell m$, where $\ell\geq 0$ and $(m,p)=1$. If
\begin{enumerate}
\item $m\mid q-1$ or
\item $m=1$ or
\item $n=r$ or $n=r^2$ for some prime $r$,
\end{enumerate}
then $\F_{q^n}$ is completely basic over $\F_q$ and there exists a primitive complete normal element of $\F_{q^n}$ over $\F_q$.
\end{corollary}
The above results imply that $\F_{q^n}$ is completely basic over $\F_q$ when $n\leq 5$ or $m=1$.
They also imply Theorems~\ref{thm:our1} and \ref{thm:our0} for $q\leq 5$. This is straightforward to check for Theorem~\ref{thm:our1}, as our cases of interest are $q>n$ and the case $n\leq 5$ is already settled. For Theorem~\ref{thm:our0} we demonstrate the case $q=5$ and we notice that the cases $q\leq 4$ are proven in a similar way. Write $n=5^\ell m$, where $(m,5)=1$. Since our cases of interest are $m<q$, we get that $m=1,2,3$ or $4$. The cases $m=1,2$ and $4$ are covered directly from Corollary~\ref{cor:cbe}, while the case $m=3$, i.e. when $n=5^\ell 3$, satisfies the conditions of Theorem~\ref{thm:cbe}, hence in any case $\F_{5^{5^\ell m}}$ is completely basic over $\F_5$. Summing up, from now on we may assume that $q\geq 7$, $n\geq 6$ and $m>1$.

Characters and their sums play a crucial role in characterizing elements of finite fields with the desired properties and in estimating the number of elements that combine all the desired properties.
\begin{definition}
Let $\mathfrak{G}$ be a finite abelian group. A \emph{character} of
$\mathfrak{G}$ is a group homomorphism $\mathfrak{G} \to \mathbb{C}^*$. The characters of $\mathfrak{G}$ form a group under multiplication,
which is isomorphic to $\mathfrak{G}$. This group is called the \emph{dual} of
$\mathfrak{G}$ and denoted by $\widehat{\mathfrak{G}}$. Furthermore, the
character $\chi_0 : \mathfrak{G} \to \mathbb{C}^*$, where $\chi_0(g) = 1$ for
all $g\in \mathfrak{G}$, is called the \emph{trivial
character} of $\mathfrak{G}$. Finally, by $\bar\chi$ we denote the inverse of $\chi$.
\end{definition}
The finite field $\F_{q^n}$ is associated with its multiplicative and its additive group. From now on, we will call the characters of $\F_{q^n}^*$ \emph{multiplicative
characters} and the characters of $\F_{q^n}$ \emph{additive characters}.
Furthermore, we will denote by $\chi_0$ and $\psi_0$ the trivial multiplicative
and additive character respectively and we will extend the multiplicative
characters to zero with the rule
\[ \chi(0) := \begin{cases} 0, & \text{if }\chi\in\widehat{\F_{q^n}^*}
\setminus \{ \chi_0 \} , \\
1, & \text{if } \chi = \chi_0 . \end{cases} \]
A \emph{character sum} is a sum that involves characters. In this work we will use the following well-known results on character sums.
\begin{lemma}[Orthogonality relations]\label{lemma:orthogonality}
Let $\chi$ be a non-trivial character of a group $\mathfrak{G}$ and $g$ a
non-trivial element of $\mathfrak{G}$. Then
\[ \sum_{x\in\mathfrak{G}} \chi(x) = 0 \quad \text{and} \quad
\sum_{\chi\in\widehat{\mathfrak{G}}} \chi(g) = 0 . \]
\end{lemma}
\begin{lemma}[Gauss sums]\label{lemma:gauss}
Let $\chi$ be a non-trivial multiplicative character and $\psi$ be a non-trivial additive character. Then
\[
\left| \sum_{x\in\F_{q^n}} \chi(x)\psi(x) \right| = q^{n/2} .
\]
\end{lemma}
The additive and the multiplicative groups of $\F_{q^n}$ can also be seen as modules. In particular $\F_{q^n}^*$ (the multiplicative group) can be seen as a $\Z$-module and $\F_{q^n}$ (the additive group) as a $\F_{q^l}[X]$-module, where $l\mid n$, under the rules $r\circ x = x^r$ and $F\circ x = \sum_{i=0}^k F_ix^{q^{li}}$, where $r\in\Z$ and $F(X)=\sum_{i=0}^kF_iX^i\in\F_{q^l}[X]$. Since both primitive elements and  normal elements over $\F_{q^l}$ are known to exist, it follows that both modules are cyclic.

Let $q'$ be the square-free part of $q^n-1$. The characteristic function for primitive elements of $\F_{q^n}$
is given by Vinogradov's formula
\[
\omega(x) := \theta(q') \sum_{d\mid q'} \frac{\mu(d)}{\phi(d)} \sum_{\chi\in\widehat{\F_{q^n}^*},\ \ord(\chi)=d} \chi(x),
\]
where $\theta(q')=\phi(q')/q'$, $\mu$ is the M\"obius function, $\phi$ is the Euler function and the \emph{order} of a multiplicative character is defined as its multiplicative order in $\widehat{\F_{q^n}^*}$. Similarly, the characteristic function for elements of $\F_{q^n}$ that are normal over $\F_{q^l}$ is
\[
\varOmega_l(x) := \theta_l(X^{n/l}-1) \sum_{F\mid X^{n/l}-1} \frac{\mu_l(F)}{\phi_l(F)} \sum_{\psi\in\widehat{\F_{q^n}}, \ \ord_l(\psi)=F} \psi(x) ,
\]
where $\theta_l(X^{n/l}-1):= \phi_l(F_l')/q^{l\cdot\deg(F_l')}$, $F_l'$ is the square-free part of $X^{n/l}-1\in\F_{q^l}[X]$, $\mu_l$ and $\phi_l$ are the M\"obius and Euler functions in $\F_{q^l}[X]$, the first sum extends over the monic divisors of $X^{n/l}-1$ in $\F_{q^l}[X]$ and the second sum runs through the additive characters of $\F_{q^n}$ of order $F$ over $\F_{q^l}$. The \emph{order} of an additive character of $\F_{q^n}$ over $\F_{q^l}$, denoted as $\ord_l$, is defined as the lowest degree monic polynomial $G\in\F_{q^l}[X]$ such that $\psi( G\circ x ) = 1$ for all $x\in\F_{q^n}$. We note that the order of an additive character of $\F_{q^n}$ over $\F_{q^l}$ always divides $X^{n/l}-1$ in $\F_{q^l}[X]$
. Furthermore, an additive or a multiplicative character has order equal to $1$ if and only if it is the trivial character
.
It is easy to see that the above characteristic functions can be written in the following more compact form, which we will use later
\begin{eqnarray*}
  \omega(x) &=& \theta(q') \sum_{\chi\in\widehat{\F_{q^n}^*},\ \ord(\chi) \mid q'}\frac{\mu(\ord(\chi))}{\phi(\ord(\chi))} \chi(x), \\
  \varOmega_l(x) &=& \theta_l(X^{n/l}-1) 
    \sum_{\psi\in\widehat{\F_{q^n}}} \frac{\mu_l(\ord_l(\psi))}{\phi_l(\ord_l(\psi))} \psi(x).
\end{eqnarray*} 

Let $\PCN_q(n)$ be the number of primitive completely normal elements of $\F_{q^n}$ over $\F_q$
and $\CN_q(n)$ be the number of completely normal elements of $\F_{q^n}$ over $\F_q$. Let $\{1=l_1<\ldots <l_k<n\}$ be the set
of proper divisors of $n$. Since all $x\in\F_{q^n}^*$ are normal over $\F_{q^n}$, it follows that an element of $\F_{q^n}$ is
completely normal over $\F_q$ if and only if it is normal over $\F_{q^{l_i}}$ for all $i=1,\ldots ,k$.
To simplify our notation, we denote
$\bq=(X^{n/l_1}-1,\ldots,X^{n/l_k}-1)$ and $\theta(\bq)=\prod_{i=1}^k \theta_{l_i}(X^{n/l_i}-1)$. We compute
\begin{eqnarray*}\label{eq:cn}
\CN_q(n) &=& \sum_{x\in\F_{q^n}} \varOmega_{l_1}(x) \cdots \varOmega_{l_{k}}(x) \\
&=& \theta(\bq) \sum_{(\psi_1,\ldots,\psi_k)}
\prod_{i=1}^k\frac{\mu_{l_i}(\ord_{l_i}(\psi_i))}{\phi_{l_i}(\ord_{l_i}(\psi_i))}
\sum_{x\in\F_{q^n}} \psi_1\cdots\psi_k(x) ,
\end{eqnarray*}
where the sums extends over all $k$-tuples of additive characters. Noting that 
\[
\sum_{x\in\F_{q^n}} \psi_1\cdots\psi_k(x) =0,\ \ \ \mbox{ for }\ \psi_1\cdots\psi_k\neq \psi_0,
\]
we obtain
\[
  \CN_q(n) = q^n\ \theta(\bq) \sum_{\substack{(\psi_1,\ldots,\psi_k)\\ \psi_1\cdots\psi_k=\psi_0}}
\prod_{i=1}^k\frac{\mu_{l_i}(\ord_{l_i}(\psi_i))}{\phi_{l_i}(\ord_{l_i}(\psi_i))}.
\]
\section{Sufficient conditions}\label{sec:main}
In this section we prove some sufficient conditions that ensure $\PCN_q(n)>0$.
\begin{theorem}\label{thm:our2}
Let $q$ be a prime power, $n\in\N$, then
\[
|\PCN_q(n) - \theta(q')\CN_q(n)| \leq q^{n/2} W(q')W_{l_1}(F_{l_1}') \cdots W_{l_k}(F_{l_k}')\theta(q')\theta(\bq) ,
\]
where $W(q')$ is the number of positive divisors of $q'$ and $W_{l_i}(F_{l_i}')$ is the number of monic divisors of $F_{l_i}'$ in $\F_{q^{l_i}}[X]$.
\end{theorem}
\begin{proof}
Using the characteristic functions, as presented in Section~\ref{sec:prelim} we deduce that
\begin{align*}
\PCN &_q  (n) = \sum_{x\in\F_{q^n}} \omega(x) \varOmega_{l_1}(x) \cdots \varOmega_{l_{k}}(x) \\
&= \theta(q')\theta(\bq)\sum_{\chi}\sum_{(\psi_1,\ldots,\psi_k)}
 \frac{\mu(\ord(\chi))}{\phi(\ord(\chi))}\prod_{i=1}^k\frac{\mu_{l_i}(\ord_{l_i}(\psi_i))}{\phi_{l_i}(\ord_{l_i}(\psi_i))}
 \sum_{x\in\F_{q^n}} \psi_1\cdots\psi_k(x) \chi(x) \\
 &= \theta(q')\theta(\bq) (S_1+S_{2}),
\end{align*}
where the term $S_1$ is the part of the above sum for $\chi=\chi_0$,
\begin{equation} \label{eq:S1}
S_1 = \sum_{(\psi_1,\ldots,\psi_k)}
 \prod_{i=1}^k\frac{\mu_{l_i}(\ord_{l_i}(\psi_i))}{\phi_{l_i}(\ord_{l_i}(\psi_i))}
 \sum_{x\in\F_{q^n}} \psi_1\cdots\psi_k(x)
 = \frac{\CN_q(n)}{\theta(\bq)}
\end{equation}
and $S_{2}$ is the part for $\chi\neq \chi_0$,
\begin{equation}\label{eq:S2}
 S_{2} = \sum_{\chi\neq\chi_0}\sum_{(\psi_1,\ldots,\psi_k)}
 \frac{\mu(\ord(\chi))}{\phi(\ord(\chi))}\prod_{i=1}^k\frac{\mu_{l_i}(\ord_{l_i}(\psi_i))}{\phi_{l_i}(\ord_{l_i}(\psi_i))}
 \sum_{x\in\F_{q^n}} \psi_1\cdots\psi_k(x) \chi(x).
 \end{equation}
 In the last sum, note that the summations runs on multiplicative characters $\chi$ of order dividing $q'$ and may be
 restricted to additive characters of order dividing the square-free part of $X^{n/l_i}-1$, which we denoted by $F_{l_i}'$.
 For the last sum we have
 \begin{align*}
  |S_{2}| 
  &\leq \sum_{\chi\neq\chi_0}\sum_{(\psi_1,\ldots,\psi_k)}
 \frac{1}{\phi(\ord(\chi))}\prod_{i=1}^k\frac{1}{\phi_{l_i}(\ord_{l_i}(\psi_i))}
 \left|\sum_{x\in\F_{q^n}} \psi_1\cdots\psi_k(x) \chi(x)  \right| \\
  &\leq q^{n/2} \sum_{\chi\neq\chi_0}\frac{1}{\phi(\ord(\chi))} \prod_{i=1}^k \sum_{\psi_i}\frac{1}{\phi_{l_i}(\ord_{l_i}(\psi_i))} \\
  &= q^{n/2} (W(q')-1) \prod_{i=1}^k W_{l_i}(F_{l_i}'),
 \end{align*}
 where we used Lemma~\ref{lemma:orthogonality} and Weil's bound, as seen in Lemma~\ref{lemma:gauss}, for the second inequality. The result follows.
\qed\end{proof}
\begin{remark}
The sieving techniques of Cohen and Huczynska \cite{cohenhuczynska03,cohenhuczynska10} could be applied here, albeit only on the multiplicative part. The potential advantage of implementing them would be reducing the number of pairs $(q,n)$ that we rely on the examples given in \cite{morganmullen96} (see Table~\ref{table:exceptions}). However, since this number is already small, the current simpler approach was favored.
\end{remark}
From the above, it is clear that we will also need the following lemma:
\begin{lemma}\label{lemma:w(r)}
For any $r\in\N$, $W(r) \leq c_{r,a} r^{1/a}$, where $c_{r,a}=2^s/(p_1 \cdots
p_s)^{1/a}$ and $p_1,\ldots ,p_s$ are the primes $\leq 2^a$ that divide $r$. In particular, $c_{r,4}<4.9$, $c_{r,8}<4514.7$
 for all $r\in\N$ and $c_{r,4}<2.9$, $c_{r,8}<2461.62$ and $c_{r,12}<5.61\cdot 10^{23}$ if $r$ is odd.
\end{lemma}
\begin{proof}
It is clear that it suffices to prove the above for $r$ square-free. Assume that $r=p_1\cdots p_s q_1\cdots q_t$, where $p_1,\ldots, p_s,q_1, \ldots ,q_t$ are distinct primes and $p_i\leq 2^a$ and $q_j>2^a$. We have that
\[ W(r) = 2^{s+t} = 2^s \cdot \underbrace{2\cdots 2}_{t \text{ times}} = 2^s (\underbrace{2^a \cdots 2^a}_{t \text{ times}})^{1/a} \leq 2^s (q_1\ldots q_t)^{1/a} = c_{r,a} r^{1/a} . \]
The bounds for $c_{r,a}$ can be easily computed.
\qed\end{proof}
Also, in order to apply the results of this section, we need a lower bound for $\CN_q(n)$.
\begin{proposition}\label{prop:cn-bound}
Let $q$ be a power of the prime $p$ and $n\in\N$, then
\[
\CN_q(n) \geq q^n \left(1 - \sum_{d|n}\left(1 - \frac{\phi_d(X^{n/d}-1)}{q^n} \right)\right)
\]
In particular, for every $n$ and $p$, we have that
\begin{equation}\label{eq:ip1}
\CN_q(n) \geq q^n \left(1-\frac{n(q+1)}{q^2}\right) ,
\end{equation}
while for $n=p^\ell m$, with $\ell\geq 1$ and $(m,p)=1$, we get
\begin{align}
\CN_q(n) & \geq q^{n} \left(1-m\left(\frac{1}{q}+\frac{1}{q^2}+\frac{1}{q^p}+\frac{4}{q^{2p}}\right)\right),\text{ for } p>2 \label{eq:ip2} \\
\CN_q(n) & \geq q^{n} \left(1-m\left(\frac{1}{q}+\frac{1}{q^2}+\frac{2}{3q^3}+\frac{3}{q^4}\right)\right),
\text{ for } p=2 . \label{eq:ip3}
\end{align}
\end{proposition}
\begin{proof}
For $d|n$, the number of elements of $\F_{q^n}$ that are normal over $\F_{q^d}$ is equal to
$\phi_d(X^{n/d}-1)$. Therefore, the number of elements of $\F_{q^n}$ that are {\it not} 
completely normal over $\F_{q}$ is at most $\sum_{d|n}(q^n - \phi_d(X^{n/d}-1))$. The 
first bound follows.

For the bound of Eq.~\eqref{eq:ip1}, we observe that
\[
\phi_d(X^{n/d}-1) = q^n \prod_{P}\left(1-\frac{1}{q^{d \deg(P)}}\right) \geq q^n \left(1-\frac{1}{q^d}\right)^{n/d},
\]
where the product extends over the prime factors of $X^{n/d}-1$ in $\F_{q^d}[X]$. Substituting in the first bound
we obtain
\[
  \CN_q(n) \geq q^n \left(1 - \sum_{d|n}\left(1 - \left(1-\frac{1}{q^d}\right)^{n/d}\right)\right) 
  \geq q^n \left(1 - \sum_{d|n}\frac{n}{dq^d}\right).
\]
The result follows upon noting that
\[
 \sum_{d|n}\frac{n}{d q^d}
  \leq n\left(\frac{1}{q} + \sum_{\substack{d|n \\ d>1}} \frac{1}{d q^d}\right)
  \leq n \left(\frac{1}{q} +\frac{1}{2} \sum_{d=2}^n q^{-d}\right) 
  \leq nq^{-2}(q+1),
\]
since $\sum_{d=2}^n q^{-d}\leq 2 q^{-2}$.

We next move on to the next two inequalitites. Let $n=p^{\ell}m$, with $(m,p)=1$. Then the by the first bound we have
\[
\CN_q(p^{\ell}m) \geq q^{p^{\ell}m} \left(1 - \sum_{j=0}^{\ell}\sum_{d|m}
\left(1 - \frac{\phi_{p^jd}(X^{p^{\ell-j}m/d}-1)}{q^{p^{\ell}m}} \right)\right).
\]
Since $p$ is the characteristic of $\F_{q}$, $X^{p^{\ell-j}m/d}-1=(X^{m/d}-1)^{p^{\ell-j}}$, and we may compute 
\[
\phi_{p^jd}(X^{p^{\ell-j}m/d}-1)\geq q^{p^{\ell}m}\left(1-\frac{1}{q^{dp^j}}\right)^{m/d}\geq 
 q^{p^{\ell}m}\left(1-\frac{m}{dq^{dp^j}}\right).
\]
Therefore, 
\begin{equation}\label{eq:main-cn-bound}
\CN_q(p^{\ell}m) \geq q^{p^{\ell}m} \left(1 - m\sum_{j=0}^{\ell}\sum_{d|m} \frac{1}{dq^{dp^j}}\right).
\end{equation}
First, we consider the case $p>2$. We compute,
\[
\sum_{d|m}\frac{1}{dq^{dp^j}}
  \leq \frac{1}{q^{p^j}} + \frac{1}{2} \sum_{d=2}^m \frac{1}{q^{dp^j}}
  \leq \frac{1}{q^{p^j}} + \frac{1}{q^{2p^j}}.
\]
Substituting in Eq.~\eqref{eq:main-cn-bound}, and using the fact that
\[
\sum_{j=0}^{\ell} \frac{1}{q^{p^j}} \leq \frac{1}{q} + \frac{1}{q^p} + \frac{2}{q^{2p}},
\]
we have
\begin{align*}
\CN_q(p^{\ell}m) & \geq q^{p^{\ell}m} \left(1 - m\sum_{j=0}^{\ell}\left(\frac{1}{q^{p^j}} + \frac{1}{q^{2p^j}} \right)\right) \\
 & \geq q^{p^{\ell}m} \left(1 - m\left(\frac{1}{q}+\frac{1}{q^p}+\frac{2}{q^{2p}}+\frac{1}{q^2}+\frac{2}{q^{2p}}\right) \right),
\end{align*}
which immediately yields Eq.~\eqref{eq:ip2}.

In the case $p=2$, we note that $m$ is odd. Then
\[
\sum_{d|m}\frac{1}{dq^{d2^j}}
  \leq \frac{1}{q^{2^j}} + \frac{1}{3} \sum_{d=3}^m \frac{1}{q^{d2^j}}
  \leq \frac{1}{q^{2^j}} + \frac{2}{3q^{3\cdot 2^j}}.
\]
Substituting in Eq.~\eqref{eq:main-cn-bound}, and using the bound
\[
\sum_{j=0}^{\ell} \frac{1}{q^{2^j}} \leq \frac{1}{q}+\frac{1}{q^2}+\frac{2}{q^4},
\]
we obtain 
\begin{align*}
\CN_q(2^{\ell}m) & \geq 
q^{2^{\ell}m} \left(1 - m\sum_{j=0}^{\ell}\left(\frac{1}{q^{2^j}} + \frac{2}{3q^{3\cdot 2^j}} \right)\right) \\
 & \geq q^{2^{\ell}m} \left(1 - m\left(\frac{1}{q}+\frac{1}{q^2}+\frac{2}{q^4}+
                    \frac{2}{3q^3} + \frac{4}{3q^{6}}\right) \right).
\end{align*}
The last bound of the statement follows.
\qed\end{proof}
We note that the bound of Eq.~\eqref{eq:ip1} is meaningful for $q\geq n+1$ and the ones in Eqs.~\eqref{eq:ip2} and \eqref{eq:ip3} are meaningful for $q> m>1$, with the sole exception of $q=4$ and $m=3$. This covers all our cases of our interest that are not covered directly by Corollary~\ref{cor:cbe}.
%
%
\section{Proof of Theorem~\ref{thm:our1}}\label{sec:proof}
In this section, we use the theory developed earlier to prove Theorem~\ref{thm:our1}. All the described computations were performed with the {\sc SageMath} software.
From Theorem~\ref{thm:our2}, we get $\PCN_q(n)>0$ provided that
\begin{equation}\label{eq:ip_pcn1}
\CN_q(n) > q^{n/2} W(q') \prod_{i=1}^k W_{l_i}(F_{l_i}') \theta_{l_i}(F_{l_i}').
\end{equation}
Clearly, $\theta_{l_i}(F_{l_i}')<1$ for all $i$ and $W_{l_i}(F_{l_i}') \leq 2^{n/l_i}$, so we have that
\[
\prod_{i=1}^k W_{l_i}(F_{l_i}') \theta_{l_i}(F_{l_i}') < 2^{\sum_{i=1}^k n/l_i} = 2^{t(n)-1}.
\]
Plugging this and Eq.~\eqref{eq:ip1} of Proposition~\ref{prop:cn-bound} into Eq.~\eqref{eq:ip_pcn1}, it suffices
to show that
\begin{equation}\label{eq:cond1}
q^{n/2}\left( 1-\frac{n(q+1)}{q^2} \right) \geq W(q') 2^{t(n)-1} .
\end{equation}
We combine the above with Lemma~\ref{lemma:w(r)}, applied for $a=8$, and a sufficient condition for $\PCN_q(n)>0$ would be
\begin{equation}\label{eq:cond2}
q^{3n/8}\left( 1-\frac{n(q+1)}{q^2} \right) \geq 4514.7 \cdot 2^{t(n)-1} .
\end{equation}
By Robin's theorem \cite{robin84},
\[
t(n) \leq e^\gamma n\log\log n + \frac{0.6483n}{\log\log n} , \ \forall n\geq 3 ,
\]
where $\gamma$ is the Euler-Mascheroni constant, therefore the condition of Eq.~\eqref{eq:cond2} becomes
\[
q^{3n/8}\left( 1-\frac{n(q+1)}{q^2}\right) > 4514.7 \cdot 2^{n\left( \log\log n\cdot e^{0.578}+ \frac{0.6483}{\log\log n}\right)-1}.
\]
Since the cases $q=n$ and $q=n+1$ are already settled by Corollary~\ref{cor:cbe}, we check the above for $q\geq n+2$ and verify that it holds for $n>1212$.

A quick computation shows that, within the range $2\leq n\leq 1212$, Eq.~\eqref{eq:cond2} is satisfied for all but 42 values of $n$, if we substitute $q$ by the least prime power greater or equal to $n+1$, $t(n)$ by its exact value and we exclude the cases when $n$ is a prime or a square of a prime, as in those cases $n$ Theorem~\ref{thm:our1} is implied by Corollary~\ref{cor:cbe}. For those values for $n$, we compute the smallest prime power $q$ that satisfies Eq.~\eqref{eq:cond2}, where $t(n)$ is replaced by its exact value. The results are presented in Table~\ref{tab:1}. 
In this region, there is a total of 1162 pairs $(n,q)$ to deal with.
\begin{table}[hbt]
\begin{center}
\begin{tabular}{|lll||lll||lll||lll|}
\hline
$n$ & $q_0$ & $q_1$  & $n$ & $q_0$ & $q_1$  & $n$ & $q_0$ & $q_1$ & $n$ & $q_0$ & $q_1$ \\ \hline\hline
6& 8& 1259&
8& 11& 431&
10& 13& 223&
12& 16& 419\\
14& 16& 107&
15& 17& 79&
16& 19& 137&
18& 23& 179\\
20& 23& 139&
21& 23& 49&
22& 25& 59&
24& 27& 243\\
26& 29& 49&
27& 29& 41&
28& 31& 89&
30& 32& 173\\
32& 37& 79&
34& 37& 41&
36& 41& 193&
40& 43& 113\\
42& 47& 121&
44& 47& 61&
45& 47& 49&
48& 53& 191\\
50& 53& 59&
54& 59& 97&
56& 59& 81&
60& 64& 256\\
66& 71& 83&
72& 79& 211&
80& 83& 101&
84& 89& 181\\
90& 97& 163&
96& 101& 163&
108& 113& 151&
120& 125& 311\\
132& 137& 139&
144& 149& 211&
168& 173& 229&
180& 191& 311\\
240& 243& 343&
360& 367& 439& & & & & & \\ \hline
\end{tabular}
\end{center}
\caption{Values for $2\leq n\leq 984$ that are not primes or square of primes, not satisfying Eq.~\eqref{eq:cond2} for $q_0$, the least prime power $\geq n+2$, where $q_1$  stands for the least prime power satisfying Eq.~\eqref{eq:cond2} for that $n$.}\label{tab:1}
\end{table}

By combining Eq.~\eqref{eq:ip_pcn1} and the first bound of Proposition~\ref{prop:cn-bound}, we get another condition, namely
\begin{equation}\label{eq:cond3}
q^{n/2}\left( 1 - \sum_{d\mid n}\left( 1-\frac{\phi_d(X^{n/d}-1)}{q^n} \right) \right) > W(q') \prod_{i=1}^k W_{l_i}(F_{l_i}') \theta_{l_i}(F_{l_i}') .
\end{equation}
By using the estimate $W(q') \leq c_{q',16} q^{n/16}$ from Lemma~\ref{lemma:w(r)} and Eq.~\eqref{eq:cond3} the list is furtherer reduced to a total of 47 pairs, if we compute all appearing quantities explicitly, in particular the constant $c_{q',16}$ is computed exactly for each value of $q'$ of interest. The list can be shrinked even more, to a total of 37 pairs, if we replace $W(q')$ by its exact value in Eq.~\eqref{eq:cond3}. These pairs $(n,q)$ are
\begin{center}
$(6, 8)$, $(6, 9)$, $(6, 11)$, $(6, 13)$, $(6, 16)$, $(6, 17)$, $(6, 19)$, $(6,
23)$, $(6, 25)$, $(6, 29)$, $(6, 31)$, $(6, 37)$, $(6, 43)$, $(6, 49)$, $(6, 61)$, $(8,
11)$, $(8, 13)$, $(8, 17)$, $(8, 19)$, $(8, 25)$, $(12, 17)$, $(12, 19)$, $(12, 23)$,
$(12, 25)$, $(12, 29)$, $(12, 31)$, $(12, 37)$, $(12, 41)$, $(12, 43)$, $(12, 49)$,
$(12, 61)$, $(12, 73)$, $(18, 37)$, $(24, 29)$, $(24, 37)$, $(24, 41)$ and $(24, 49)$.
\end{center}
However, 24 of those pairs correspond to completely basic extensions, as we directly check the conditions of Theorem~\ref{thm:cbe}. The remaining 13 pairs are
\begin{center}
$(6, 8)$,
$(6, 11)$,
$(6, 17)$,
$(6, 23)$,
$(6, 29)$,
$(8, 11)$,
$(8, 19)$,
$(12, 17)$,
$(12, 23)$,
$(12, 29)$,
$(12, 41)$,
$(24, 29)$ and
$(24, 41)$
\end{center}
and they all satisfy $q\leq 97$ or $q^n<10^{50}$ and examples of primitive completely normal elements are given in \cite{morganmullen96}.
This completes the proof of Theorem~\ref{thm:our1}.
\section{Proof of Theorem~\ref{thm:our0}}\label{sec:proof-0}
Let $n=p^{\ell}m$, with $(m,p)=1$ and $m<q$. Our goal is to prove that $\PCN_q(p^{\ell}m)>0$ for $\ell\geq 0$ and $m\geq 1$. First, we notice that if $\ell=0$, then Theorem~\ref{thm:our1} implies Theorem~\ref{thm:our0}, hence we only need to consider the case $\ell\geq 1$.
Also, the case $m=1$ is settled by Corollary~\ref{cor:cbe}, so from this point on, we assume that $\ell\geq 1$ and $m\geq 2$.

The set of proper divisors of $n$ is $\{p^j d\ :\ 0\leq j\leq \ell,\ d|m\}\setminus\{n\}$. From Theorem~\ref{thm:our2},
we get $\PCN_q(n)>0$ provided that
\begin{equation}\label{eq:ip_pcn2}
\CN_q(n) > q^{n/2} W(q') \prod_{\substack{j=0,\ldots,\ell,\ d\mid m \\ (j,d)\neq (\ell ,m)}} W_{p^jd}(F_{p^jd}') \theta_{p^jd}(F_{p^jd}'),
\end{equation}
where $F_{p^jd}' = X^{m/d}-1$. 
Clearly, $\theta_{p^jd}(F_{p^jd}')<1$ for all $0\leq j\leq\ell$, $d|m$ and 
$W_{p^jd}(X^{m/d}-1) \leq 2^{m/d}$, so we have
\[
\prod_{\substack{j=0,\ldots,\ell,\ d\mid m \\ (j,d)\neq (\ell ,m)}} W_{p^jd}(F_{p^jd}') \theta_{p^jd}(F_{p^jd}') < 2^{(\ell+1)\sum_{d|m} m/d - 1} = 2^{(\ell+1)t(m)-1},
\]
where $t(m)$ denotes the sum of divisors of $m$.
So a sufficient condition is
\begin{equation}\label{eq:cond-1}
\CN_q(n)\geq q^{n/2}W(q') 2^{(\ell+1)t(m)-1}.
\end{equation}
\paragraph{The $p>2$ case}
For $p>2$, using Eq.~\eqref{eq:ip2} from Proposition~\ref{prop:cn-bound} and Lemma~\ref{lemma:w(r)}, applied for $a=8$, we obtain the sufficient condition,
\begin{equation} \label{eq:cond-2}
q^{3p^{\ell}m/8}\left(1-m\left(\frac{1}{q}+\frac{1}{q^2}+\frac{1}{q^p}+\frac{4}{q^{2p}}\right)\right)
\geq 2257.35 \cdot 2^{(\ell+1)t(m)}.
\end{equation}
The RHS of Eq.\eqref{eq:cond-2} does not depend on $q$, while the LHS is an increasing function of $q$, so, since we are interested in $q>m$ and the case $q=m+1$ or $q\leq 5$ is settled by Corollary~\ref{cor:cbe}, 
it suffices to prove Eq.~\eqref{eq:cond-2} for $q=\max(m+2,7)$. First we consider the case $m=2$. A short calculation shows that
the condition becomes
\[
q^{3p^{\ell}/4}\left(1-2\left(\frac{1}{q}+\frac{1}{q^2}+\frac{1}{q^p}+\frac{4}{q^{2p}}\right)\right)
\geq 2257.35 \cdot 2^{3(\ell+1)},
\]
which is satisfied for $q\geq 7$ and $\ell \geq 1$, with the exceptions 
$(\ell,m,q)=(1,2,7)$ or $(1,2,9)$.

For $m\geq 3$, we upper bound $t(m)$ by Robin's theorem \cite{robin84},
\[
t(m) \leq e^\gamma m\log\log m + \frac{0.6483m}{\log\log m} , \ \forall m\geq 3 ,
\]
where $\gamma$ is the Euler-Mascheroni constant.
Furthermore, for $m\leq q-1$ we have 
\[ 1-m\left( \frac{1}{q}+\frac{1}{q^2}+\frac{1}{q^p}+\frac{4}{q^{2p}} \right) \geq \frac{m}{q^4} \] 
and the condition becomes
\[
mq^{3p^{\ell}m/8-4}\geq 2257.35 \cdot 2^{(\ell+1)\left(e^\gamma m\log\log m + \frac{0.6483m}{\log\log m}\right)}.
\]
Since $q\geq m+2$ and $p\geq 3$, it suffices to show that
\[
m(m+2)^{3^{\ell+1}m/8-4} \geq 2257.35 \cdot 2^{(\ell+1)\left(e^\gamma m\log\log m + \frac{0.6483m}{\log\log m}\right)} ,
\]
which is true for $\ell\geq 4$ and $m\geq 3$. The inequality is violated for the following 54 pairs
\[
  (\ell,m) = (1 ,3\leq m\leq 49), (2, 3\leq m\leq 8), (3,3) .
\]
For those pairs $(\ell,m)$ we go back to Eq.~\eqref{eq:cond-2}, and check for which prime powers $q$ it is violated. Since 
the LHS is an increasing function of $q$, this process will produce one more exceptional triple $(\ell,m,q)$, namely $(1,7,9))$. So, in total there are 3 exceptional triples $(\ell,m,q)$,
\[
(1,2,7),(1,2,9),(1,7,9),
\]
but only $(1,7,9)$ corresponds to a non completely basic extension, by Theorem~\ref{thm:cbe}.
For this case, an example of primitive completely normal basis is given in \cite{morganmullen96}.
\paragraph{The $p=2$ case}
For $p=2$, the argument is nearly identical, the only difference being the choice of the bound of Eq.~\eqref{eq:ip3} from Proposition~\ref{prop:cn-bound}.
Since $m>1$ and $(m,2)=1$ we consider only $m\geq 3$, while from Corollary~\ref{cor:cbe}, it suffices to work with $q\geq 8$.

In Eq.~\eqref{eq:cond-1}, $q'$ is odd and an application of Lemma~\ref{lemma:w(r)} for $a=8$ yields
\[
W(q') \leq 2461.62 \cdot (q')^{1/8} \leq 2461.62 \cdot q^{n/8} .
\]
Using this, Eq.~\eqref{eq:ip3} from Proposition~\ref{prop:cn-bound} and Eq.~\eqref{eq:cond-1}, we obtain the sufficient condition
\begin{equation} \label{eq:cond-3}
q^{3\cdot 2^{\ell}m/8}\left(1-m\left(\frac{1}{q}+\frac{1}{q^2}+\frac{2}{3q^3}+\frac{3}{q^4}\right)\right)
\geq 2461.62 \cdot 2^{(\ell+1)t(m)-1}.
\end{equation}
Using Robin's bound and the fact that
\[ 1-m\left( \frac{1}{q} + \frac{1}{q^2}+ \frac{2}{3q^3} +\frac{3}{q^4} \right)\geq \frac{m}{12 q^3} \] we obtain the condition
\[
m\cdot q^{3\cdot 2^{\ell}m/8-3} 
\geq 6\cdot 2461.62 \cdot 2^{(\ell+1)\left(e^\gamma m\log\log m + \frac{0.6483m}{\log\log m}\right)}.
\]
Since $q\geq m+2$, it suffices to show that
\[
m\cdot (m+2)^{3\cdot 2^{\ell}m/8-3} 
\geq 6\cdot 2461.62 \cdot 2^{(\ell+1)\left(e^\gamma m\log\log m + \frac{0.6483m}{\log\log m}\right)}, \text{ for } m\geq 7
\]
and 
\[
m\cdot 8^{3\cdot 2^{\ell}m/8-3} 
\geq 6\cdot 2461.62\cdot 2^{(\ell+1)\left(e^\gamma m\log\log m + \frac{0.6483m}{\log\log m}\right)}, \text{ for } 3\leq m\leq 5.
\]
The above inequalities are true for $\ell\geq 2$ with the exceptions
\[
  (\ell ,m)  = (2, 3\leq m\leq 157), (3 , 3\leq m\leq 19), (4,3),(4, 5) ,
  (5 ,3), (6,3).
\]
For those pairs $(\ell,m)$ we go back to Eq.~\eqref{eq:cond-3}, and check for which values of $q\geq 8$ that are powers of
2 the condition is violated. Those triples $(\ell,m,q)$ are
\[
(2, 3, 8), (2, 3, 16), (2, 5, 8), (2, 7, 8), (3,3,8),
\]
but only $(2,3,8)$, $(2,5,8)$ and $(3,3,8)$ correspond to a non-completely basic extension, by Theorem~\ref{thm:cbe}.
Examples of primitive completely normal bases for the corresponding extensions are given in \cite{morganmullen96}.
 
 The remaining cases to check are for $\ell=1$, $m\geq 3$ and $q\geq 8$.
 An application of Lemma~\ref{lemma:w(r)} for $a=12$ yields
 \[ W(q') \leq 5.61\cdot 10^{23} \cdot (q')^{1/12} \leq 5.61\cdot 10^{23} \cdot q^{n/12} . \]
 Using this, Eq.~\eqref{eq:ip3} from Proposition~\ref{prop:cn-bound} and Eq.~\eqref{eq:cond-1}, we obtain the sufficient condition
 \[
 q^{5\cdot 2\cdot m/12}\left(1-m\left(\frac{1}{q}+\frac{1}{q^2}+\frac{2}{3q^3}+\frac{3}{q^4}\right)\right)
\geq \frac{5.61\cdot 10^{23}}{2} \cdot 2^{2t(m)}.
 \]
Using Robin's bound it is sufficient to show that for $q\geq m+2$, and $q\geq 8$,
\[
q^{5m/6}\frac{m}{12q^3} \geq 2.81\cdot 10^{23}\cdot 4^{e^\gamma m\log\log m + \frac{0.6483m}{\log\log m}}.
\]
 This leads to the conditions
 \begin{align*}
   m8^{5m/6-3} &\geq 12\cdot 2.81\cdot 10^{23}\cdot 4^{e^\gamma m\log\log m + \frac{0.6483m}{\log\log m}}, \text{ for } 3\leq m\leq 5, \\
   m(m+2)^{5m/6-3}&\geq 12\cdot 2.81\cdot 10^{23}\cdot 4^{e^\gamma m\log\log m + \frac{0.6483m}{\log\log m}}, \text{ for } m\geq 7.
 \end{align*}
 The conditions are satisfied for $m>873$. For $3\leq m\leq 873$, we obtain the following 116 exceptions
 \begin{align*}
 (\ell,m ,q) =& (1, 3, 2^3\leq q \leq 2^{34}), (1, 5, 2^3\leq q\leq 2^{21}), (1, 7, 2^4\leq q\leq 2^{16}), \\
 & (1,9, 2^4\leq q\leq 2^{13}), (1, 11, 2^4\leq q\leq 2^{11}), (1, 13, 2^4\leq q\leq 2^9), \\
 & (1, 15, 2^5\leq q\leq 2^{10}),(1, 17, 2^5\leq q\leq 2^8), (1, 19, 2^5\leq q\leq 2^7), \\
 & (1, 21, 2^5\leq q\leq 2^8), (1, 23, 2^5\leq q\leq 2^6),(1, 25, 2^5\leq q\leq 2^6),\\
 & (1, 27, 2^5\leq q\leq 2^7), (1, 29, 2^5), (1, 33, 64), (1, 35, 64), (1, 45, 64).
 \end{align*}
 After removing the triples which are covered by Corollary~\ref{cor:cbe}, that is where $m\mid q-1$, the list shrinks to the following list of 85 possible exceptions:
\begin{center}
$(1, 3, 8)$, $(1, 3, 32)$, $(1, 3, 128)$, $(1, 3, 512)$, $(1, 3, 2048)$, $(1,
3, 8192)$, $(1, 3, 32768)$, $(1, 3, 131072)$, $(1, 3, 524288)$, $(1, 3,
2097152)$, $(1, 3, 8388608)$, $(1, 3, 33554432)$, $(1, 3, 134217728)$, $(1, 3,
536870912)$, $(1, 3, 2147483648)$, $(1, 3, 8589934592)$, $(1, 5, 8)$, $(1, 5,
32)$, $(1, 5, 64)$, $(1, 5, 128)$, $(1, 5, 512)$, $(1, 5, 1024)$, $(1, 5, 2048)$,
$(1, 5, 8192)$, $(1, 5, 16384)$, $(1, 5, 32768)$, $(1, 5, 131072)$, $(1, 5,
262144)$, $(1, 5, 524288)$, $(1, 5, 2097152)$, $(1, 7, 16)$, $(1, 7, 32)$, $(1, 7,
128)$, $(1, 7, 256)$, $(1, 7, 1024)$, $(1, 7, 2048)$, $(1, 7, 8192)$, $(1, 7,
16384)$, $(1, 7, 65536)$, $(1, 9, 16)$, $(1, 9, 32)$, $(1, 9, 128)$, $(1, 9, 256)$,
$(1, 9, 512)$, $(1, 9, 1024)$, $(1, 9, 2048)$, $(1, 9, 8192)$, $(1, 11, 16)$, $(1,
11, 32)$, $(1, 11, 64)$, $(1, 11, 128)$, $(1, 11, 256)$, $(1, 11, 512)$, $(1, 11,
2048)$, $(1, 13, 16)$, $(1, 13, 32)$, $(1, 13, 64)$, $(1, 13, 128)$, $(1, 13,
256)$, $(1, 13, 512)$, $(1, 15, 32)$, $(1, 15, 64)$, $(1, 15, 128)$, $(1, 15,
512)$, $(1, 15, 1024)$, $(1, 17, 32)$, $(1, 17, 64)$, $(1, 17, 128)$, $(1, 19,
32)$, $(1, 19, 64)$, $(1, 19, 128)$, $(1, 21, 32)$, $(1, 21, 128)$, $(1, 21, 256)$,
$(1, 23, 32)$, $(1, 23, 64)$, $(1, 25, 32)$, $(1, 25, 64)$, $(1, 27, 32)$, $(1, 27,
64)$, $(1, 27, 128)$, $(1, 29, 32)$, $(1, 33, 64)$, $(1, 35, 64)$ and $(1, 45, 64)$
\end{center}
Each of the above triples, with the sole exception of $(1,3,8)$, satisfy the initial condition of Eq.~\eqref{eq:ip_pcn2}, but the latter is already covered by the examples provided in \cite{morganmullen96}. This concludes the proof of Theorem~\ref{thm:our0}.
\begin{table}[h]
\begin{center}
\begin{tabular}{|ll||ll||ll||ll||ll||ll|}
\hline
$n$ & $q$ & $n$ & $q$ & $n$ & $q$ & $n$ & $q$ & $n$ & $q$ & $n$ & $q$ \\
\hline \hline
$6$ & $ 8$ &
$6$ & $ 11$ &
$6$ & $ 17$ &
$6$ & $ 23$ &
$6$ & $ 29$ &
$8$ & $ 11$ \\
$8$ & $ 19$ &
$12$ & $ 17$ &
$12$ & $ 23$ &
$12$ & $ 29$ &
$12$ & $ 41$ &
$24$ & $ 29$ \\
$24$ & $ 41$ &
$21$ & $9$ &
$12$ & $8$ &
$20$ & $8$ &
$24$ & $8$ &
$6$ & $8$ \\ \hline
\end{tabular}
\end{center}
\caption{Pairs $(n,q)$ that were not dealt with theoretically.\label{table:exceptions}}
\end{table}
\section{Conclusions}
In this work, a step towards the proof of Conjecture~\ref{conj:mm} was taken. We note that our restrictions $q<n$ and $q<m$ are direct consequences of the lower bounds for $\CN_q(n)$ from Proposition~\ref{prop:cn-bound}. For our methods to work more generally, new bounds for $\CN_q(n)$ are required. We believe that this would be an interesting and challenging direction for further research.
\section*{Acknowledgments}
The authors are grateful to the anonymous referee for his/her comments that significantly improved the quality and readability of this paper.

\begin{thebibliography}{10}

\bibitem{blessenohljohnsen86}
D.~Blessenohl and K.~Johnsen.
\newblock Eine versch\"arfung des satzes von der normalbasis.
\newblock {\em J. Algebra}, 103(1):141--159, 1986.

\bibitem{cohenhachenberger99}
S.~D. Cohen and D.~Hachenberger.
\newblock Primitive normal bases with prescribed trace.
\newblock {\em Appl. Algebra Engrg. Comm. Comput.}, 9(5):383--403, 1999.

\bibitem{cohenhuczynska03}
S.~D. Cohen and S.~Huczynska.
\newblock The primitive normal basis theorem {{--}} without a computer.
\newblock {\em J. London Math. Soc.}, 67(1):41--56, 2003.

\bibitem{cohenhuczynska10}
S.~D. Cohen and S.~Huczynska.
\newblock The strong primitive normal basis theorem.
\newblock {\em Acta Arith.}, 143(4):299--332, 2010.

\bibitem{diffiehellman76}
W.~Diffie and M.~Hellman.
\newblock New directions in cryptography.
\newblock {\em IEEE Trans. Information Theory}, 22(6):644--654, 1976.

\bibitem{gao93}
S.~Gao.
\newblock {\em Normal Basis over Finite Fields}.
\newblock PhD thesis, University of Waterloo, 1993.

\bibitem{hachenberger97}
D.~Hachenberger.
\newblock {\em Finite Fields: Normal Bases and Completely Free Elements},
  volume 390 of {\em Kluwer Internat. Ser. Engrg. Comput. Sci.}
\newblock Kluwer Academic Publishers, Boston, MA, 1997.

\bibitem{hachenberger10}
D.~Hachenberger.
\newblock Primitive complete normal bases: {E}xistence in certain 2-power
  extensions and lower bounds.
\newblock {\em Discrete Math.}, 310(22):3246--3250, 2010.

\bibitem{hachenberger13}
D.~Hachenberger.
\newblock Completely normal bases.
\newblock In G.~L. Mullen and D.~Panario, editors, {\em Handbook of Finite
  Fields}, pages 128--138. CRC Press, Boca Raton, 2013.

\bibitem{hachenberger16}
D.~Hachenberger.
\newblock Asymptotic existence results for primitive completely normal elements
  in extensions of {G}alois fields.
\newblock {\em Des. Codes Cryptogr.}, 80(3):577--586, 2016.

\bibitem{hsunan11}
C.~Hsu and T.~Nan.
\newblock A generalization of the primitive normal basis theorem.
\newblock {\em J. Number Theory}, 131(1):146--157, 2011.

\bibitem{kapetanakis13}
G.~Kapetanakis.
\newblock An extension of the (strong) primitive normal basis theorem.
\newblock {\em Appl. Algebra Engrg. Comm. Comput.}, 25(5):311--337, 2014.

\bibitem{kapetanakis14}
G.~Kapetanakis.
\newblock Normal bases and primitive elements over finite fields.
\newblock {\em Finite Fields Appl.}, 26:123--143, 2014.

\bibitem{lenstraschoof87}
H.~W. Lenstra, Jr and R.~J. Schoof.
\newblock Primitive normal bases for finite fields.
\newblock {\em Math. Comp.}, 48(177):217--231, 1987.

\bibitem{lidlniederreiter97}
R.~Lidl and H.~Niederreiter.
\newblock {\em Finite Fields}.
\newblock Cambridge University Press, Cambridge, second edition, 1997.

\bibitem{morganmullen96}
I.~H. Morgan and G.~L. Mullen.
\newblock Completely normal primitive basis generators of finite fields.
\newblock {\em Util. Math.}, 49:21--43, 1996.

\bibitem{robin84}
G.~Robin.
\newblock Grandes valeurs de la fonction somme des diviseurs et hypoth\`ese de
  {R}iemann.
\newblock {\em J. Math. Pures Appl.}, 63(2):187--213, 1984.

\end{thebibliography}
\end{document}